\documentclass[12pt, draft]{article}
\usepackage{fixmath}

\usepackage[T1]{fontenc}
\usepackage{ucs}
\usepackage[utf8x]{inputenc}
\usepackage[english]{babel}

\usepackage{microtype}
\usepackage{babel}

\usepackage{authblk}

\usepackage{mathtools} 

\usepackage{enumerate}

\usepackage{geometry} 

\usepackage{amsfonts, amsmath, amsthm, amssymb}
\usepackage{mathrsfs}
\usepackage{hyperref} 

\usepackage{tikz}
\usetikzlibrary{shapes}
\tikzstyle{node}=[draw, circle, minimum size=5pt, fill=black, inner sep=0pt]

\newtheorem{theorem}{Theorem}

\newtheorem{lemma}{Lemma}

\theoremstyle{definition}

\theoremstyle{remark}
\newtheorem{remark}{Remark}

\DeclareMathOperator{\GEM}{GEM}

\DeclareMathOperator{\labelset}{\mathfrak L}

\providecommand{\lift}[1]{#1^{\mathsf L}}

\providecommand{\Prob}[1]{\mathbb{P}\left\{#1\right\}}

\providecommand{\card}[1]{\texttt{\#}#1}

\providecommand{\bracket}[1]{\langle#1\rangle}

\newcommand{\res}[2]{#1|_#2}

\title{Lifting linear preferential attachment trees\\ yields the arcsine coalescent}
\author{Helmut H.~Pitters}
\affil{Department of Statistics\\University of California, Berkeley}

\begin{document}


\maketitle


\begin{abstract}
We consider linear preferential attachment trees which are specific scale-free trees also known as (random) plane-oriented recursive trees. Starting with a linear preferential attachment tree of size $n$ we show that repeatedly applying a so-called lifting yields a continuous-time Markov chain on linear preferential attachment trees. Each such tree induces a partition of $\{1, \ldots, n\}$ by placing labels in the same block if and only if they are attached to the same node in the tree. Our main result is that this Markov chain on linear preferential attachment trees induces a partition valued process which is equal in distribution (up to a random time-change) to the arcsine $n$-coalescent, that is the multiple merger coalescent whose $\Lambda$ measure is the arcsine distribution.

\end{abstract}

\section{Introduction}
A linear preferential attachment tree $T$ of size $n$ is a random planar rooted tree on $n$ nodes labeled $1, \ldots, n$ such that the labels along any non-backtracking path starting from the root are increasing. The tree $T$ can be constructed by attaching nodes as they arrive in the order of increasing labels as follows.
\begin{enumerate}
  \item Start with a root node labeled $1.$
  \item At step $n-1$ we have a (random) tree on $n-1$ nodes with labels $1, \ldots, n-1.$ A new node with label $n$ is added to the existing tree, namely it is attached with an edge to node $v$ with probability proportional to
  \[d^+(v)+1,\]
where $d^+(v)$ counts the number of successors of $v$.
There are $d^+(v)+1$ available positions for an additional successor of $v,$ since $T$ is planar. Node $n$ is assigned to one of these positions chosen uniformly at random.
\end{enumerate}

Let us now introduce a continuous-time Markov chain on linear preferential attachment trees. The Markov cain starts in state $T$ and its absorbing state is the tree consisting of only one node. If $T$ has $n>1$ nodes, after waiting an exponential time of rate $n,$ among all nodes of $T$ choose one node $U$ uniformly at random. If $U$ is a leaf, do nothing. Otherwise, select one of $U$s successors, $V$ say, uniformly at random. We now lift the edge $\{U, V\}$ as follows. Collect the labels attached to vertices in the subtree $T_V$ rooted at $V$ and attach them to $U,$ then remove the edge $\{U, V\}$ together with $T_V$. We call this procedure ``lifting'' of an edge, following Berestycki's lecture notes~\cite{Berestycki2009}. However, in the literature variants of this procedure are sometimes called ``cutting'' or ``pruning''. The tree $\lift{T}$ obtained after lifting $T$ is again a linear preferential attachment tree, as we show in Lemma~\ref{lem:lemma1}. Moreover, $\lift{T}$ is labeled by the blocks of some partition $\pi$ of $[n].$ We say that $\pi$ is induced by $\lift{T}$. A moment's thought shows that, in our context, the appropriate way to order the blocks of this partition is according to their least element. Applying this lifting procedure repeatedly yields a Markov chain taking values in the set of linear preferential attachment trees which are labeled by the blocks of partitions of $[n]$. If we start with the linear preferential attachment tree $T$ of size $n$ and only keep track of the partitions of $[n]$ induced by the lifted trees, we obtain a stochastic process with state space the set $\mathscr P_{[n]}$ of partitions of $[n]$ such that the partitions become coarser and coarser as time passes. In fact, starting with a linear preferential attachment tree of size $n$ our main result, Theorem~\ref{thm:lifting_lpats}, shows that by repeated lifting we obtain (up to a random time change) the arcsine $n$-coalescent, which is the multiple merger coalescent corresponding to the measure $\Lambda$ taken to be the arcsine distribution. See section~\ref{sec:main_results} for a definition of multiple merger coalescent processes.

The first construction of a coalescent process by a similar lifting procedure applied to random recursive trees was given by Goldschmidt and Martin~\cite{GoldschmidtMartin2005}. The authors start with a random recursive tree and show that the partition-valued process induced by repeated lifting yields the Bolthausen-Sznitman $n$-coalescent corresponding to the $\Lambda$ measure given by the uniform distribution. Abraham and Delmas give another construction of the beta$(\frac{3}{2}, \frac{1}{2})$ $n$-coalescent by lifting random binary trees in~\cite{AbrahamDelmas2013}, and a construction of the jump chain of the beta$(1+\alpha, 1-\alpha)$ $n$-coalescent by lifting stable Galton-Watson trees in~\cite{AbrahamDelmas2015}. It should however be noted that the lifting procedures employed in these examples differ from each other.

\section{Main Results}\label{sec:main_results}
An \emph{increasing tree} on the labels $1, \ldots, n$ is a rooted tree on $n$ nodes which are labeled by $1, \ldots, n$ such that any sequence of labels along any non-backtracking path starting at the root is increasing. A \emph{plane-oriented recursive tree (PORT)} is a planar increasing tree, i.e.~the successors of any node are ordered.
A PORT on $n$ nodes can be constructed recursively as follows.
\begin{enumerate}
  \item Start with the tree $t_{1}$ consisting only of the root with label $1,$ which trivially is a PORT.
  \item Given a PORT $t_{n-1}$ on $n-1$ nodes pick a node $v$ in $t_{n-1}$ and put a further node with label $n$ into any of the $d^+(v)+1$ positions available at $v.$
\end{enumerate}

Denote by $\card A$ the cardinality of a set $A$. We slightly abuse notation and write $v\in t$ if $v$ is a node of $t$. All $n-1$ nodes of $t_{n-1}$ except for the root are successors of some node, thus there are
\[\sum_{v\in t_{n-1}} (d^{+}(v)+1) = n-2+n-1 = 2n-3=2(n-1)-1\]
PORTs on $n$ nodes that can be constructed by adding node $n$ to $t_{n-1}.$ If $t_n, t_n'$ are PORTs of size $n$ constructed in this way from PORTs $t_{n-1},$ respectively $t_{n-1}',$ of size $n-1,$ then $t_{n-1}\neq t_{n-1}'$ clearly implies $t_n\neq t_n'$. Consequently, letting $\mathfrak P_{n}$ denote the \emph{set of plane-oriented recursive trees on $n$ labeled nodes}, we just showed that $\card\mathfrak P_{n}=(2(n-1)-1)\card\mathfrak P_{n-1},$ hence
\begin{align}\label{eq:ports}
  \card\mathfrak P_{n} = 1\cdot 3\cdots(2(n-1)-1) = (2(n-1)-1)!! = \frac{1}{2^{n-1}}\frac{(2(n-1))!}{(n-1)!}=\frac{n!}{2^{n-1}}C_{n-1},
\end{align}
where for any integer $n\geq -1$ the \emph{double factorial} is defined by
\begin{align}
  n!!\coloneqq \begin{cases}
    1\cdot 3\cdot 5\cdots (n-2)n & \text{if $n$ is odd,}\\
    2\cdot 4\cdot 6\cdots (n-2)n & \text{if $n$ is even,}\\
    1 & \text{if } n\in \{-1, 0\},
  \end{cases}
\end{align}
and $C_n\coloneqq (2n)!/(n!(n+1)!),$ $n\in\mathbb N_0,$ denotes the \emph{$n$th Catalan number}.


A \emph{linear preferential attachment tree (LPAT) of size $n$} is an element $T_{n}$ of $\mathfrak P_{n}$ drawn uniformly at random. At times we write LPAT$(n)$  for ``LPAT of size $n$,'' respectively PORT$(n)$ for ``PORT of size $n$''. Notice that the recursive construction of a PORT$(n)$ immediately yields a recursive construction of an LPAT($n$) $T_n$ by picking a node $v$ in an LPAT($n-1$) with probability proportional to $d^+(v)+1,$ and attaching a new node with label $n$ to $v.$

\begin{remark} (i) Fix two natural numbers $m, n$ with $m\leq n.$ Define the map $\rho_{nm}$ from $\mathfrak P_n$ to $\mathfrak P_m,$ which we call~\emph{restriction}, as follows. If $t_n\in\mathfrak P_n$ is a plane-oriented recursive tree of size $n,$ let $\rho_{nm}(t_n)$ be the subtree in $t_n$ spanned by the nodes whose labels are smaller than or equal to $m.$ If $T_n$ is an LPAT($n$) and $T_m$ is an LPAT($m$), it follows from the recursive construction of linear preferential attachment trees that
\begin{align}
  \rho_{nm}(T_n) =_d T_m.
\end{align}

  (ii) Consider a sequence $\{T_k, 1\leq k\leq n\}$, where $T_k$ is an LPAT of size $k,$ and $T_{k+1}$ is obtained from $T_k$ by the recursive (random) construction we just described. Clearly, $\{T_k\}$ is a Markov chain. Moreover, $T_k$ contains all the information about the past $\{T_1, \ldots, T_{k}\},$ since obviously $T_j,$ $j\leq k,$ is the subtree in $T_k$ spanned by the nodes with labels $1, \ldots, j,$ and $T_k$ is therefore a representation of the $\sigma$-algebra generated by the $T_1, \ldots, T_k.$
\end{remark}

A \emph{partition} of a nonempty set $A$ is a set, $\pi$ say, of nonempty pairwise disjoint subsets of $A$ whose union is $A$. The members of $\pi$ are called the \emph{blocks} of $\pi.$ Let $\mathscr P_A$ denote the set of all partitions of $A$.

In what follows we want somewhat more flexibility in the labeling of trees. Namely, we want to label the nodes in a tree by blocks $B$ of a partition of $[n].$ To this end we endow any partition $\pi$ of $[n]$ by the \emph{order of least elements}, denoted $\leq,$ namely let $B\leq C$ if and only if $\min B\leq \min C$ for any two blocks $B, C$ in $\pi.$ If $t$ is a tree whose nodes are labeled by the blocks in $\pi$ we call $\labelset (t)\coloneqq \pi$ the \emph{label set of $t.$} With this definition a \emph{plane-oriented recursive tree on $\pi=\{B_1, \ldots, B_k\}\in\mathscr P_n,$} PORT($\pi$) for short, respectively a \emph{linear preferential attachment tree on $\pi,$} LPAT($\pi$) for short, is a plane-oriented recursive tree, respectively linear preferential attachment tree on $k$ nodes which are labeled by the blocks $B_1, \ldots, B_k$. Denote by $\mathfrak P_\pi$ the \emph{set of all plane-oriented recursive trees on $\card\pi$ nodes labeled by the blocks $B_1, \ldots, B_k.$}

\begin{remark} (i) Fix two natural numbers $m, n$ such that $m\leq n.$ Fix a partition $\pi\in\mathscr P_n,$ and let $\pi'$ be the restriction of $\pi$ to $[m]$. Moreover, define the map $\rho_{\pi m}\colon \mathfrak P_\pi\to\mathfrak P_\pi'$ as follows. If $T$ is an LPAT($\pi$), consider the subtree $\bar T$ of $T$ spanned by the nodes whose labels (which are subsets of $[n]$) contain an element in $[m].$ Moreover, let $T'$ be obtained from $\bar T$ by restricting the labels in $\bar T$ to $[m],$ i.e.~by replacing any label $B$ of a node $v\in\bar T$ by $B\cap [m].$ We set
\begin{align}
  \rho_{\pi m}(T) \coloneqq T',
\end{align}
and call $T'$ the \emph{restriction of $T$ to $m.$} Notice that, again from the recursive construction of linear preferential attachment trees, for any LPAT($\pi$) $T$ and any LPAT$(\pi')$ $T'$ we have
\begin{align}
  \rho_{\pi m}(T) =_d T'.
\end{align}

  (ii) Notice that in our nomenclature a linear preferential attachment tree of size $n$ is a tree chosen uniformly at random among all plane-oriented recursive trees of size $n.$ In particular, a PORT is always a deterministic tree, whereas for any $n\geq 3$ an LPAT$(n)$ is a random tree. However, in the literature the reader may occasionally find that the term PORT is used for both, deterministic plane-oriented recursive trees as we define them, as well as for linear preferential attachment trees. Other names for LPATs that appear in the literature are heap-ordered trees, nonuniform recursive trees and scale-free trees, cf.~\cite{Hwang2007}.
\end{remark}

We now define the operation of lifting which is at the heart of our construction of the arcsine coalescent. Consider a rooted tree $t$ on $n$ labeled nodes.~\emph{Lifting an edge $e=\{u, v\}$ in $t$} works as follows: Assume that $u$ is closer (in graph distance) to the root than $v.$ Then attach the labels on the subtree $t_v$ rooted at $v$ to $u,$ discard both $t_v$ and $e,$ and only keep track of the subtree containing the root of $t.$ In what follows we will choose the edge that is to be lifted in a particular and random fashion. Namely, we pick a node $U$ in $t$ uniformly at random, and, provided $U$ is not a leaf, we pick one of $U$s successors, call it $V,$ uniformly at random. By \emph{lifting a tree t} we mean picking an edge $\{U, V\}$ randomly in the manner described above and than lifting $\{U, V\}$ in $t$. We denote by $\lift{t}$ the tree that is obtained by lifting $t.$

Our first observation is that if we lift $\{U, V\}$ in an LPAT $T_n$ of size $n$ we obtain an LPAT $\lift{T_n}$ on the new label set $\labelset (\lift{T_n})$.

\begin{lemma}\label{lem:lemma1}
Let $T_{n}$ be an LPAT of size $n$ and let $\lift{T_{n}}$ be the tree obtained by lifting $T_{n}$. Then, conditionally on $\lift{T_{n}}$ having label set $\pi,$ $\lift{T_{n}}$ is an LPAT($\pi$).
\end{lemma}
\begin{proof}
Fix a subset $C\subseteq [n]$ of size $k\coloneqq \card C\geq 2$. Let $t_{C}$ denote an arbitrary but fixed plane-oriented recursive tree of size $n-k+1$ with label set $\bracket{C;n},$ which is the partition of $[n]$ whose blocks consist of the elements in $[n]\setminus C$ and the block $C.$ Then
  \begin{align*}
    \Prob{\lift{T_{n}} = t_{C}\;|\;\lift{T_{n}} \text{ has label set }\bracket{C; n}} &= \frac{\Prob{\lift{T_{n}} = t_{C}}}{\Prob{\lift{T_{n}} \text{ has label set }\bracket{C; n}}}\\
    &= \card\mathfrak P_{n-k+1}^{-1},
  \end{align*}
where the last equality is seen as follows. Slightly abusing notation, we write $d_{t_C}^+(C)$ for the number of successors of the node in $t_C$ that carries label $C.$
By the recursive construction there are $(d_{t_C}^+(C)+1)\card\mathfrak P_{k-1}$ PORTs of size $n$ whose subtree consisting of the first $n-k+1$ nodes agrees with $t_C$ (where we identify the labels $c\coloneqq\min C$ and $C$), each of which is equally likely to be observed under $T_n$, i.e.
  \begin{align}\label{eq:lifting}
    \Prob{\lift{T_{n}} = t_{C}} &= \Prob{\lift{T_{n}} = t_{C}|T_n\text{ contains }t_C}\Prob{T_n\text{ contains }t_C}\\\notag
    &= \frac{1}{n(d_{t_C}^+(C)+1)}\frac{(d_{t_C}^+(C)+1)\card\mathfrak P_{k-1}}{\card\mathfrak P_{n}}=\frac{1}{n}\frac{\card\mathfrak P_{k-1}}{\card\mathfrak P_{n}}.
  \end{align}
To lift the one edge in $T_n$ which connects the nodes with labels $c$ and $c'\coloneqq C\setminus \{c\}$ ($c'$ is the root of the subtree with label set $C\setminus \{c\}$) provided $T_n$ contains $t_C,$ one first has to pick the node with label $c,$ which happens with probability $1/n$
and, conditionally on $c$ being picked, one has to pick the successor of $c$ with label $c',$ which happens with probability $1/(d_{t_C}^+(C)+1)$. Consequently, the probability to pick the one edge in $T_n$ that yields $t_C$ after being lifted is $1/(n(d_{t_C}^+(C)+1))$. It is important to notice that the probability in~\eqref{eq:lifting} only depends on $C$ via $\card C$.

Moreover,
  \begin{align*}
    \Prob{\lift{T_{n}} \text{ has label set }\bracket{C; n}} &= \sum_{t_C'\in \mathfrak P_{\bracket{C; n}}}\Prob{\lift{T_n}=t_C'}\\
    &= \frac{1}{n}\frac{\card\mathfrak P_{n-k+1}\card\mathfrak P_{k-1}}{\card\mathfrak P_{n}}.
  \end{align*}
The claim follows.
\end{proof}
To summarize, starting with an LPAT of size $n$ repeated lifting yields a Markov chain on $\bigcup_{\pi\in\mathscr P_n} \mathcal L_{\pi},$ where $\mathcal L_\pi$ denotes the set of linear preferential attachment trees on $\pi.$
We now modify this lifting process to obtain a continuous-time Markov chain $\{\mathscr T_{n}(t), t\geq 0\}$ as follows. Start with initial state $\mathscr T_{n}(0)=T_{n},$ where $T_n$ is an LPAT of size $n$. If $\mathscr T_{n}$ is in state $T,$ attach an exponential clock to each node $v$ of $T$ that rings at rate $1,$ all clocks being independent and independent of $T$. When the first clock rings, if the corresponding node $V$ is a leaf, do nothing. Otherwise, pick a successor $U$ of $V$ uniformly at random and lift the edge $\{U, V\}$ in $T$ to obtain the next state $\lift{T}$.

Before we state our main result, we recall the notion of multiple merger coalescent processes which were introduced independently by Donnelly and Kurtz~\cite{donnelly1999}, Pitman~\cite{MR1742892} and Sagitov~\cite{MR1742154}. A \emph{(standard) multiple merger $n$-coalescent} $\Pi_{n}$ is a continuous-time Markov chain with state space $\mathscr P_{[n]},$ the set of all partitions of $[n]\coloneqq \{1, \ldots, n\},$ and initial state $\Delta_{n}\coloneqq \{\{1\}, \{2\}, \ldots, \{n\}\}$ such that if $\Pi_{n}$ is in a state of $b$ blocks any $1\leq k\leq b$ specific blocks merge at rate
\begin{align}\label{eq:coalescent_rates}
  \lambda_{b, k}\coloneqq \int_{0}^{1}x^{k-2}(1-x)^{b-k}\Lambda(dx),
\end{align}
where $\Lambda$ is a finite measure on the unit interval. The process $\Pi_n$ is also referred to as the $\Lambda$ $n$-coalescent. The integral formula~\eqref{eq:coalescent_rates} for the transition rates is due to Pitman~\cite{MR1742892} and follows from the requirement that the $\Lambda$ $n$-coalescents be consistent as $n$ varies. Here consistency means that for each $m\leq n$ the restriction of $\Pi_n$ to $[m]$ is equal in distribution to $\Pi_m$. In particular, there exists a process $\Pi\coloneqq \{\Pi(t), t\geq 0\},$ the so-called \emph{$\Lambda$ coalescent}, with state space the partitions of the positive integers $\mathbb N\coloneqq \{1, 2, \ldots \}$ such that for any $n\in\mathbb N$ the restriction of $\Pi$ to $[n]$ is equal in distribution to $\Pi_n$.

For $a, b>0$ the beta$(a, b)$ coalescent is the multiple merger coalescent corresponding to $\Lambda$ the \emph{beta distribution with parameters a, b} on $(0, 1)$ with density
\begin{align}
  x\mapsto\frac{x^{a-1}(1-x)^{b-1}}{B(a, b)}\mathbf{1}_{(0, 1)}(x)\qquad (x\in\mathbb R)
\end{align}
where $B(a, b)\coloneqq \int_{0}^{1}x^{a-1}(1-x)^{b-1}dx=  \Gamma(a)\Gamma(b)/\Gamma(a+b)$ denotes the \emph{beta integral} and for $x>0$ the \emph{gamma function} is defined by $\Gamma(x)\coloneqq \int_0^\infty t^{x-1}e^{-t}dt$. Denoting by $\lambda_{n, k}(a, b)$ the infinitesimal rates of the beta$(a, b)$ coalescent, equation~\eqref{eq:coalescent_rates} implies
\begin{align}\label{eq:coalescent_rates_beta}
  \lambda_{n, k}(a, b) = \frac{B(n-k+b, k-2+a)}{B(a, b)}.
\end{align}

Since the beta distribution with parameters $\frac{1}{2}, \frac{1}{2}$ and density
\begin{align}
  x\mapsto \frac{1}{\pi\sqrt{x(1-x)}}\mathbf{1}_{(0, 1)}(x)\qquad (x\in\mathbb R) 
\end{align}
is the \emph{arcsine law}, we call the corresponding multiple merger coalescent the \emph{arcsine coalescent}. By $\mathbf{1}_A(x)$ we denote the indicator of a set $A$ which equals one if $x\in A$ and zero otherwise. Using $\Gamma(n+\frac{1}{2})=\sqrt{\pi}(2n)!/(4^nn!)$ and~\eqref{eq:coalescent_rates} we find the transition rates of the arcsine coalescent to be given by
\begin{align}\label{eq:arcsine_coalescent_transition_rates}
  \lambda_{n, k}(\frac{1}{2}, \frac{1}{2}) &= \frac{B(k-2+\frac{1}{2}, n-k+\frac{1}{2})}{B(\frac{1}{2}, \frac{1}{2})}=\frac{1}{\pi}\frac{\Gamma(k-2+\frac{1}{2})\Gamma(n-k+\frac{1}{2})}{\Gamma(n-1)}\\\notag
  &= \frac{(2(k-2))!}{4^{k-2}(k-2)!}\frac{(2(n-k))!}{4^{n-k}(n-k)!} = \frac{(k-1)!(n-k+1)!}{4^{n-2}}C_{k-2}C_{n-k}
\end{align}
for $n\geq k\geq 2.$

We now turn to the stochastic process $\Pi'_n$ recording the partitions of $[n]$ induced by the process $\mathscr T_n$ of repeatedly lifting linear preferential attachment trees. More formally, $\Pi'_n\coloneqq \{\Pi'_n(t), t\geq 0\}$ is defined by letting $\Pi'_{n}(t)\coloneqq \labelset(\mathscr T_{n}(t))$. The next theorem shows that up to a random time change the process $\Pi'_{n}$ is the arcsine coalescent restricted to $[n]$.

\begin{theorem}\label{thm:lifting_lpats}
The process $\Pi'_{n}\coloneqq \{\Pi'_{n}(t), t\geq 0\}$ defined by $\Pi'_{n}(t)\coloneqq \labelset(\mathscr T_{n}(t))$ is a continuous-time Markov chain with state space the partitions of $[n],$ initial state $\Delta_{n}$ and absorbing state $\{[n]\}$ such that whenever $\Pi'_{n}$ is in a state of $b$ blocks a merger of $k$ blocks occurs at rate
\begin{align}
  \lambda'_{b,k} = \frac{2^{b-2}(b-2)!}{\card\mathfrak P_{b}}\lambda_{b,k}(\frac{1}{2}, \frac{1}{2}) =\frac{1}{2(b-1)bC_{b-1}}\lambda_{b,k}(\frac{1}{2}, \frac{1}{2}) \qquad (b\geq k\geq 2).
\end{align}
\end{theorem}

\begin{proof}
Because of Lemma~\ref{lem:lemma1} it suffices to compute the transition rates of $\Pi'_{n}$ in its initial state $\Delta_{n}$. Fix a subset $C\subseteq [n]$ of size $k\coloneqq \card C\geq 2.$ Let $\bracket{C; n}$ denote the partition of $[n]$ consisting of the singletons in $[n]\setminus C$ and the block $C.$ Recall that $\mathscr T_n$ starts in state $\mathscr T_n(0)=T_n,$ where $T_n$ is an LPAT($n$). The rate at which we see a lifted tree $\lift{T_{n}}$ whose label set consists of $C$ and the elements in $[n]\setminus C$ is given by
\begin{align}\label{eq:proof_rates}
  \lambda_{n, k}'(C)= \frac{\card\mathfrak P_{n-k+1}(d_{t_C}^+(C)+1)\card\mathfrak P_{k-1}}{\card\mathfrak P_{n}}\frac{1}{n(d_{t_C}^+(C)+1)}n
  = \frac{\card\mathfrak P_{n-k+1}\card\mathfrak P_{k-1}}{\card\mathfrak P_{n}}.
\end{align}
In particular, $\lambda_{n, k}'$ only depends on $C$ via $\card C=k.$ For this reason we drop the argument $C$ and write $\lambda_{n,k}'$. The first equality in~\eqref{eq:proof_rates} holds since there are $\card\mathfrak P_{n-k+1}$ PORTs on $\bracket{C; n}$, and $\card\mathfrak P_{k-1}$ PORTs on the partition of $C\setminus\{c\}$ into singletons, where $c\coloneqq \min C$. Moreover, any PORT$(n)$ built by choosing an element from each of these sets of trees and joining their nodes labeled $C,$ respectively $c,$ by an edge in $d_{t_C}^+(C)+1$ different ways only yields $\lift{T_{n}}=t_C$ if we lift this particular edge, which happens with probability $1/(n(d_{t_C}^+(C)+1)).$ Finally, the overall rate at which we see an event happen when $\Pi'_n$ is in a state consisting of $n$ blocks is $n,$ hence the last factor.

Recall (cf.~\cite[Exercise 13.1.14, p.~609]{Arfken2001}) that the double factorial with odd arguments can be expressed in terms of the gamma function as
\begin{align}\label{eq:double_factorial_gamma}
  \Gamma(n+\frac{1}{2}) = \sqrt{\pi}\frac{(2n-1)!!}{2^{n}}\qquad (n\in\mathbb N_0).
\end{align}
Use~\eqref{eq:ports} and~\eqref{eq:double_factorial_gamma} to rewrite $\lambda'_{n, k}$ as
\begin{align}
  \lambda'_{n,k} &=\frac{\card\mathfrak P_{n-k+1}\card\mathfrak P_{k-1}}{\card\mathfrak P_{n}}\\\notag
  &= \frac{(2(n-k)-1)!!(2(k-2)-1)!!}{\card\mathfrak P_{n}}\\\notag
  &= \Gamma(n-k+\frac{1}{2})\frac{2^{n-k}}{\sqrt\pi}\Gamma(k-\frac{3}{2})\frac{2^{k-2}}{\sqrt\pi}\frac{1}{\card\mathfrak P_{n}}\\\notag
  &= \frac{2^{n-2}}{\pi}\frac{(n-2)!}{\card\mathfrak P_{n}}\frac{\Gamma(n-k+\frac{1}{2})\Gamma(k-\frac{3}{2})}{\Gamma(n-1)}\\\notag
  &= \frac{2^{n-2}}{\pi}\frac{(n-2)!}{\card\mathfrak P_{n}}B(k-\frac{3}{2}, n-k+\frac{1}{2}),\\\notag
  \intertext{and, recalling the transition rates of the arcsine coalescent in~\eqref{eq:arcsine_coalescent_transition_rates},}
  \lambda_{n,k}' &= \frac{2^{n-2}}{\pi}\frac{(n-2)!}{\card\mathfrak P_{n}}B(\frac{1}{2}, \frac{1}{2})\lambda_{n,k}(\frac{1}{2}, \frac{1}{2})\\\notag
  &= \frac{2^{n-2}(n-2)!}{\card\mathfrak P_{n}}\lambda_{n,k}(\frac{1}{2}, \frac{1}{2})= \frac{(n-1)!(n-2)!}{2(2(n-1))!}\lambda_{n,k}(\frac{1}{2}, \frac{1}{2}).
\end{align}
\end{proof}

We now turn to the process recording the limiting frequency as $n\to\infty$ of the block in $\Pi_n$ that contains $1$.

\section{The block containing $1$}
In order to better understand the behaviour of the block in $\Pi'_n$ that contains the label $1$ we recall the notion of an exchangeable partition and the Chinese Restaurant Process.
A relabeling of a partition $\pi$ of $[n]$ according to some permutation $\sigma$ of $[n]$ is the partition $\sigma\pi$ consisting of the blocks
\[\sigma B\coloneqq \{\sigma(b)\colon b\in B\} \qquad (B\in\pi).\]
A random partition $\bar\Pi$ of $[n]$ is called \emph{exchangeable} if its distribution is invariant under any relabeling, i.e.~if for any permutation $\sigma$ of $[n]$ one has
\[\sigma\bar\Pi=_{d}\bar\Pi.\]
The partitions we have encountered so far that are induced by lifting LPATs are clearly exchangeable. There are various constructions of exchangeable partitions, and the construction via the Chinese Restaurant Process will turn out to be useful for our purposes. To review the Chinese Restaurant Process we closely follow Pitman's lecture notes~\cite{Pitman2006}.

The \emph{Chinese Restaurant Process}, first introduced in Pitman and Dubins~cite{}, is a discrete-time Markov chain whose state at time $m$ is a random permutation $\sigma_{m}$ of $[m].$ The cycles of $\sigma_{m}$ constitute a (random) partition $\bar\Pi_{m}$ of $[m],$ and these random partitions are \emph{consistent} as $m$ varies, that is for each $m$ and $l\leq m$ the restriction $\res{\bar\Pi_{m}}{{[l]}}$ of $\bar\Pi_{m}$ to $[l]$ is equal in distribution to $\bar\Pi_{l}.$ Here the restriction $\res{\pi}{B}$ of any partition $\pi$ of $A$ to a subset $B\subseteq A$ is the partition of $B$ consisting of the non-empty blocks $C\cap B$ where $C$ ranges over all blocks in $\pi.$ Here we focus on a special case of the Chinese Restaurant Process parameterized by a pair of real numbers $(\alpha, \theta)$ such that $0\leq\alpha\leq 1$ and $\theta>-\alpha.$ Picture then a restaurant with an unlimited number of empty tables numbered $1, 2, \ldots,$ each capable of seating an unlimited number of customers. Customers arrive one by one, they are numbered in the order of arrival, and take seats according to the following rule: customer $1$ sits at table $1.$ Suppose $m$ customers already arrived and together occupy the first $k\in\mathbb N$ tables with $m_{i}\geq 1$ customers sitting at table $1\leq i\leq k$. The next customer $m+1$ chooses to sit
\begin{enumerate}
  \item at table $i$ with probability $(m_{i}-\alpha)/(m+\theta),$ where he chooses his left neighbor among the customers at table $i$ uniformly at random,
  \item alone at the $(k+1)$th table with probability $(\theta+k\alpha)/(m+\theta).$
\end{enumerate}
If we assign the integers in $1, \ldots, m$ to the same block according to whether or not the corresponding customers sit at the same table, we obtain a random partition of $[m].$ Denote this partition by $\bar\Pi_{m}.$ By Kingman's theory of exchangeable partitions there exists a partition $\bar\Pi$ of the positive integers $\mathbb N$ such that for any $m\in\mathbb N$ the restriction of $\bar\Pi$ to $[n]$ is equal in distribution to $\bar\Pi_{m}$  and for each block $B\in\bar\Pi$ its \emph{asymptotic frequency}
\begin{align}
  \lim_{m\to\infty}\frac{\card(B\cap [m])}{m}
\end{align}
exists almost surely. Moreover, it is known that these limiting frequencies in size-biased order of least elements have the representation
\begin{align}
  (\tilde P_{1}, \tilde P_{2}, \ldots) =_{d} (W_{1}, \bar W_{1}W_{2}, \bar W_{1}\bar W_{2}W_{3}, \ldots),
\end{align}
where the $(W_{i})_{i\geq 1}$ are independent, $W_{i}$ is governed by a beta$(1-\alpha, \theta+i\alpha)$ distribution, and $\bar W_{i}\coloneqq 1-W_{i}$. The distribution of $(\tilde P_{1}, \tilde P_{2}, \ldots)$ is the so-called \emph{Griffiths-Engen-McCloskey distribution} with parameters $(\alpha, \theta),$ denoted $\GEM(\alpha, \theta).$ The distribution of $(P_{1}, P_{2}, \ldots ),$ defined by ranking the $\tilde P_{1}, \tilde P_{2}, \ldots$ in decreasing order, is the so-called \emph{Poisson-Dirichlet distribution} with parameters $(\alpha, \theta).$

There is a Chinese Restaurant Process sitting in an LPAT$(n)$ that we now turn to. We now study another partition $\pi(t)$ (not to be confused with the label set $\labelset(t)$ of $t$) of $[n]$ induced by a tree $t$ with nodes labeled by a partition $\pi$ of $[n].$ Let $\rho$ denote the root of $t,$ and for any node $v$ in $t$ let $t_v$ denote the subtree in $t$ rooted at $v.$ Define two labels $i,$ $j\in [n]$ to be in the same block of $\pi(t)$ if and only if one of the subtrees of $t$ rooted at a successor of $\rho$ contains both $i$ and $j,$ more precisely, if $i, j\in \{m\colon m\in B\in \labelset(t_v)\}$ for some successor $v$ of $\rho$.

From the recursive construction of LPATs it is immediate that if $T_{n+1}$ is an LPAT of size $n+1$ and $T_{n}$ is an LPAT of size $n,$ the restriction of $\pi(T_{n+1})$ to $[n]$ is equal in distribution to $\pi(T_{n})$ (i.e.~the partitions $\pi(T_{n+1})$ and $\pi(T_{n})$ are said to be consistent).

Suppose then that $T_{n}$ is an LPAT of size $n$ with label set $\Delta_{[n]}$ such that its root has $k$ successors each of which subtends a subtree of size $n_{i}\geq 1,$ in particular $\sum_{i}n_{i}=n.$ In order to obtain $T_{n+1},$ we attach the label $n+1$ to a node $v$ in $T_{n}$ with probability proportional to $d^{+}(v)+1.$ In terms of partitions, the new node $n+1$
\begin{enumerate}
  \item creates a new block w.p.~$\frac{k+1}{2n+1}=\frac{k\frac{1}{2}+\frac{1}{2}}{n+\frac{1}{2}},$
  \item joins a block of size $n_{i}$ w.p.~$\frac{2n_{i}-1}{2n+1}=\frac{n_{i}-\frac{1}{2}}{n+\frac{1}{2}}$.
\end{enumerate}
This shows that $\pi(T_{n})$ has the same distribution as the partition of $[n]$ induced by the $(\frac{1}{2}, \frac{1}{2})$-Chinese Restaurant Process. A more general result revealing the Chinese Restaurant Process in a preferential attachment tree is given in~\cite[Proposition 3]{KubaPanholzer2014} (notice however that Kuba and Panholzer call these trees generalized plane-oriented recursive trees). The basic properties of the Chinese Restaurant Process reviewed earlier imply that as $n\to\infty,$ the asymptotic frequencies $(\tilde P_{1}, \tilde P_{2}, \ldots)$ of the descendencies of the successors of the root (in order of their appearance) follow a $\GEM(\frac{1}{2}, \frac{1}{2})$ distribution.

Consider now the asymptotic frequency
\begin{align}
  F'(t)\coloneqq \lim_{n\to\infty}\frac{\card\{1\leq i\leq n\colon \text{$i$ and $1$ are in the same block of $\Pi'(t)$}\}}{n}
\end{align}
of the block in $\Pi'_{n}(t)$ containing $1.$ The magnitudes of the jumps of $\{F'(t), t\geq 0\}$ are given by $(\tilde P_{i}),$ and the interarrival times between successive jumps form an i.i.d.~sequence of exponential $1$ random variables.
\\

\noindent\textbf{Acknowledgements.} The author thanks Steve Evans, Christina Goldschmidt and Martin M\"{o}hle for comments on an earlier version of this manuscript, as well as James Martin and Jim Pitman for fruitful discussions.

\bibliographystyle{amsplain}
\bibliography{literature}
\end{document}